\documentclass{article}
\usepackage[english]{babel}
\usepackage{amsmath}
\usepackage{amssymb}
\usepackage{amsthm}
\usepackage{amsfonts}
\theoremstyle{definition}
\newtheorem{definition}{Definition}[section]
\theoremstyle{proposition}
\newtheorem{proposition}{Proposition}[section]
\theoremstyle{corollary}
\newtheorem{corollary}{Corollary}[section]
\theoremstyle{theorem}
\newtheorem{theorem}{Theorem}[section]
\theoremstyle{lemma}
\newtheorem{lemma}{Lemma}[section]
\theoremstyle{definition}
\newtheorem{example}{Example}[section]
\theoremstyle{remark}

\usepackage[utf8]{inputenc}
\usepackage{pstricks}
\usepackage[a4paper, portrait, margin=1in]{geometry}

\title{Homotopical selection principles in bitopological dynamical systems}
\author{Santanu Acharjee$^{1,\dagger}$, Kabindra Goswami$^2$, 
Hemanta Kumar Sarmah$^3$ }
\date {$^{1,3}$Department of Mathematics\\
 Gauhati University\\
 Guwahati-781014, Assam, India\\
$^{2}$Department of Mathematics\\ Goalpara College\\ Goalpara-783101, Assam, India\\
E-mails: $^{1}$sacharjee326@gmail.com, $^{2}$kabindragoswami@gmail.com, $^{3}$hsarmah@gauhati.ac.in\\
$^\dagger$Corresponding author: Santanu Acharjee}

\begin{document}
\maketitle

\begin{abstract}
Homotopy deals with the intuitive idea of continuous deformation of a continuous map between two topological spaces. In this paper, we introduce homotopical selection principles in bitopological dynamical systems ($BTDS$). Here, we define $BTDS-$homotopy, $BTDS-$iteration homotopy and $BTDS-$path homotopy, and then using these notions we define various selection properties viz.  H-Rothberger property, H-Menger property, PH-Rothberger property, PH-Menger property and their weaker versions. We also discuss several results connecting these concepts in bitopological dynamical systems.
\end{abstract}
\section{Introduction}

Homotopy deals with the intuitive idea of continuous deformation of a continuous map between two topological spaces. According to \cite{40}, the idea of homotopy for the continuous maps of the unit interval was originated by Jordan \cite{39} in 1866 and that for loops was introduced by Poincaré \cite{41} in 1895. One may refer to \cite{42} for more on history of homotopy. The concept of homotopy equivalence was introduced by Hurewicz \cite{43} [see also \cite{40, 42}]. Later, many researchers studied the concept of homotopy and the study is still going on. Recently, Badiger and Venkatesh \cite{38} introduced generalised (g)-homotopy of g-continuous maps. \vskip 2mm

Kelly \cite{17} introduced the concept of bitopological
spaces from the perspective of quasi-pseudo metric spaces. Due to the flexibility of having two topologies on the same set, it became attractive to the topologists as well as researchers working in different fields. Recent applications of bitopology can be found in computer science \cite{28}, economics (\cite{26}, \cite{27}), medical sciences \cite{25}, etc. Some recent works in bitopological space can be found in \cite{29,30,31}.  \vskip 2mm

Bitopological dynamical system is a new branch in the theory of dynamical systems recently introduced by Acharjee et al. \cite{2}. It allows the study of dynamical properties of two parallel states (represented by two topologies) at a time. The application of bitopological dynamical system can be found in child birth process \cite{2}. In \cite{34}, Nada and Zohny made three conjectures regarding human embryo development. Acharjee et al. \cite{2,33,32} disproved these conjectures using bitopological dynamical system. For other works in bitopological dynamical systems, one may refer to \cite{15,35}.  \vskip 2mm

In 1924, Menger \cite{6} started the investigation of selective covering properties of topological spaces. After that Hurewicz \cite{14}, Rothberger \cite{7} and many other topologists continued the study of those covering properties. In 1996, Scheepers \cite{5} gave a general definition of selection principles by applying selection hypothesis to different open covers and originated a systematic investigation
of selection principles in topology. In recent years, selection principles in topological spaces have been studied by many researchers \cite{4,8,11,12,13}. The systematic study of selection principles in bitopological space was originated in \cite{9,37}. For recent works of selection principles in bitopological space, one may refer to \cite{3,10,22,36}.  \vskip 2mm

Till now no study has been done on homotopy from the perspective of bitopological spaces. In this paper, we introduce homotopy in bitopological dynamical systems via selection principles. First we define $BTDS-$homotopy, $BTDS-$iteration homotopy and $BTDS-$path homotopy and then using these notions we define various selection properties in bitopological dynamical systems. 

\section{Preliminaries} In this section, we procure some existing definitions from bitopological dynamical systems and selection principles which will be required in our next sections. One may refer \cite{20, 1, 21} for  definitions of homotopy,  homotopically equivalent, path and path homotopy.  Throughout the paper; $ \mathbb{R} $, $ \mathbb{N} $ and $ \mathbb{N}_0 $ denote the set of real numbers, the set of positive
integers and the set of non-negative integers respectively. If $(X,\tau)$ is a topological space, $A \subseteq X$ and $\mathcal{U}$ is a collection of subsets of $X$, then $\overline{A}$ denotes the closure of $A$, and  $\bigcup \mathcal{U}=\bigcup\{U:U\in \mathcal{U}\}$. If $ (X,\tau_{1} ,\tau_{2}) $ is a bitopological space, then $\tau_{1}-cl(A)$ denotes the closure of $A$ with respect to the topology $\tau_1$ and $\tau_{2}-cl(A)$ denotes the closure of $A$ with respect to the topology $\tau_2$.\\

\begin{definition} \cite{2}
A bitopological dynamical system is a pair $ (X,f) $, where $
(X,\tau_{1} ,\tau_{2}) $ is a bitopological space and $
f:X\rightarrow X $ is a pairwise continuous map. The dynamics is
obtained by iterating the map.

The forward orbit of a point $ x\in X $ under $ f $ is defined as $
O_{+}(x)=\{f^{n}(x):n\in \mathbb{N}_0\}, $ where $ f^{n} $ denotes the
$ n^{th} $ iteration of the map $ f $. If $ f $ is a homeomorphism,
then the backward orbit of $ x $ is the set $
O_{-}(x)=\{f^{-n}(x):n\in \mathbb{N}_0\} $ and the full orbit of $ x $
(or simply orbit of $ x $) is the set $ O(x)=\{f^{n}(x):n\in
\mathbb{Z}\} $.

Here, homeomorphism of $ f $ indicates homeomorphism of the function
$ f:(X,\tau_{i}) \rightarrow (X,\tau_{i}) $ separately for $ i
\in \{1,2\} $ as it is clear in terms of bitopological space. Thus,
we can consider pairwise homeomorphism equivalently.
\end{definition}
\begin{definition} \cite{3}
Let $X$ be a topological space. If $\mathcal{A}$ and $\mathcal{B}$ are collections of subsets of a space $X$, then the symbol $\mathcal{B}<\mathcal{A}$ denotes the fact that for each $B \in \mathcal{B}$ there is $A \in \mathcal{A}$ with $B \subset A$.
\end{definition}

In topology, selection principles have several applications. One may refer \cite{5,10,22} for fundamental notions related to selection principles in topology.\\

\begin{definition} \cite{9}
A bitopological space $(X,\tau_1 ,\tau_2)$ is said to be $\delta_2$-Menger if for each sequence $(\mathcal{U}_n : n \in \mathbb{N})$ of open covers of $(X, \tau_i)$, $i=1,2$, there is a sequence $(\mathcal{V}_n : n \in \mathbb{N})$ of finite families of $\tau_j$-open sets, $j \neq i$, such that for each $n$, $\mathcal{V}_n < \mathcal{U}_n$ and $\bigcup_{n\in \mathbb{N}} \mathcal{V}_n =X$.
\end{definition}
\begin{definition} \cite{9}
A bitopological space $(X,\tau_1 ,\tau_2)$ is said to be $(\tau_i ,\tau_j)$-weakly Menger, $i,j=1,2,i \neq j,$ if for each sequence $(\mathcal{U}_n : n \in \mathbb{N})$ of open covers of $(X, \tau_i)$ there is a sequence $(\mathcal{V}_n : n \in \mathbb{N})$ of finite families such that for each $n$, $\mathcal{V}_n \subset \mathcal{U}_n$ and $X=\tau_{j}-cl(\bigcup_{n\in \mathbb{N}} \mathcal{V}_n), j \neq i.$
\end{definition}
\begin{definition} \cite{9}
A bitopological space $(X,\tau_1 ,\tau_2)$ is said to be $(\tau_i ,\tau_j)$-almost Menger, $i,j=1,2,i \neq j,$ if for each sequence $(\mathcal{U}_n : n \in \mathbb{N})$ of open covers of $(X, \tau_i)$ there is a sequence $(\mathcal{V}_n : n \in \mathbb{N})$ of finite families such that for each $n$, $\mathcal{V}_n \subset \mathcal{U}_n$ and $X=\bigcup_{n\in \mathbb{N}}(\bigcup_{V\in \mathcal{V}_n} \tau_{j}-cl(V))$.
\end{definition}
\begin{definition} \cite{10}
A bitopological space $(X,\tau_1 ,\tau_2)$ is said to be $(i ,j)$-almost Rothberger, $i,j=1,2,i \neq j,$ if for each sequence $(\mathcal{U}_n : n \in \mathbb{N})$ of $\tau_i$-open covers of $X$, there exists a sequence $(U_n : n \in \mathbb{N})$ such that, for any $n$, $U_n \in \mathcal{U}_n$ and $X=\bigcup_{n\in \mathbb{N}} \tau_{j}-cl(U_n )$.
\end{definition}
\begin{definition} \cite{16,24}
A topological space $X$ is a $P$-space if the union of countably many closed subsets of $X$ is closed in $X$.
\end{definition}
\begin{definition} \cite{18}
A bitopological space $ (X,\tau_{1} ,\tau_{2}) $ is pairwise $T_1$ if for each pair of distinct points $x$ and $y$ of $X$ there is a $\tau_{1}$-open set $U$ and a $\tau_{2}$-open set $V$ such that $x\in U,y\notin U$ and $y \in V, x \notin V$.
\end{definition}
\begin{definition} \cite{17}
Let $ (X,\tau_{1} ,\tau_{2}) $ be a bitopological space. Then $\tau_{1}$ is said to be regular with respect to $\tau_{2}$ if, for each point $x$ in $X$ and each $\tau_{1}$-closed set $P$ such that $x \notin P$, there is a $\tau_{1}$-open set $U$ and a $\tau_{2}$-open set $V$ such that $x \in U$, $P\subseteq V$, and $U \cap V=\emptyset$. The bitopological space  $ (X,\tau_{1} ,\tau_{2}) $ is called pairwise regular if $\tau_{1}$ is regular with respect to $\tau_{2}$ and vice-versa.
\end{definition}
\begin{proposition} \cite{18}
If $ (X,\tau_{1} ,\tau_{2}) $ is a bitopological space, the following are equivalent:
\begin{enumerate}
    \item $\tau_{1}$ is regular with respect to $\tau_{2}$,
    \item for each point $x \in X$ and $\tau_{1}$-open set $G$ containing $x$ there is a $\tau_{1}$-open set $H$ such that $x\in H \subseteq \tau_{2}-cl(H) \subseteq G$,
    \item for each point $x\in X $ and $\tau_{1}$-closed set $K$ such that $x\notin K$, there is a $\tau_{1}$-open set $M$ such that $x \in M$ and $(\tau_{2}-cl(M))\cap K=\emptyset$.
\end{enumerate}
\end{proposition}
\begin{definition} \cite{18}
A bitopological space $ (X,\tau_{1} ,\tau_{2}) $ is pairwise $T_3$ if it is pairwise regular and pairwise $T_1$.
\end{definition} 
\begin{definition} \cite{19}
In a bitopological space $ (X,\tau_{1} ,\tau_{2}) $, $\tau_{1}$ is locally (countably) compact with respect to $\tau_{2}$ if each point of $X$ has a $\tau_{2}$ neighbourhood which is $\tau_{1}$ (countably) compact.
\end{definition}
We will say that a bitopological space $ (X,\tau_{1} ,\tau_{2}) $ is pairwise locally compact if $\tau_{1}$ is locally compact with respect to $\tau_{2}$ and $\tau_{2}$ is locally compact with respect to $\tau_{1}$.
\begin{proposition} \cite{19}
If $ (X,\tau_{1} ,\tau_{2}) $ is pairwise regular and pairwise Hausdorff and $\tau_{1}$ is locally compact with respect to $\tau_{2}$, then $\tau_{1}\subseteq \tau_{2}$.
\end{proposition}
\vskip 0.5cm
\section{Homotopy in Bitopological Dynamical System}

In this section, first we introduce $BDTS$-homotopy, $BDTS$-iteration homotopy and related results. We introduce an important property of a pairwise continuous map between two bitopological spaces.

\begin{lemma} Let $ (X,\tau_{1} ,\tau_{2}) $ and $ (Y,\psi_{1} ,\psi_{2}) $ be two  bitopological spaces. Also, let $
f:(X,\tau_{1} ,\tau_{2})\rightarrow (X,\tau_{1} ,\tau_{2}) $ and $g:(Y,\psi_{1} ,\psi_{2})\rightarrow (Y,\psi_{1} ,\psi_{2}) $ be two pairwise continuous maps. Then, if $
F:(X,\tau_{1} ,\tau_{2})\rightarrow (Y,\psi_{1} ,\psi_{2}) $ is a pairwise continuous map, then the compositions $F\circ f$ and $g\circ F$ are pairwise continuous maps from $(X,\tau_{1} ,\tau_{2})$ to $(Y,\psi_{1} ,\psi_{2})$.
\end{lemma}
\begin{proof} We shall denote the $\tau_{1} -$open set of $X$ as $U_i$ and $\tau_{2} -$open set of $X$ as $V_i$, where $i=1,2,3,...$. Also, we shall denote the $\psi_{1} -$open set of $Y$ as $M_j$ and $\psi_{2} -$open set of $Y$ as $N_j$, where $j=1,2,3,...$. Now, let $M_j$ be an arbitrary $\psi_{1} -$open set. Then, 
\begin{align*}
(F\circ f)^{-1}(M_j) & =(f^{-1}\circ F^{-1})(M_j)\\
& = f^{-1}(F^{-1}(M_j))\\
& =f^{-1}(U_i), \; (\text{let,} \; U_i =F^{-1}(M_j))\\.
& =U_k, \; (\text{let,} \; U_k =f^{-1}(U_i))\\
& = \text{a}\; \tau_1 -\text{open set}.
\end{align*}
and 
\begin{align*}
(g\circ F)^{-1}(M_j) & =(F^{-1}\circ g^{-1})(M_j)\\
& = F^{-1}(g^{-1}(M_j))\\
& =F^{-1}(M_l), \; (\text{let,} \; M_l =g^{-1}(M_j))\\
& =U_l, \; (\text{let,} \; U_l =F^{-1}(M_l))\\
& = \text{a}\; \tau_1 -\text{open set}.
\end{align*}
Similarly for any $\psi_{2} -$open set $N_j$, $(F\circ f)^{-1}(N_j)=$ a $\tau_{2} -$open set and $(g\circ F)^{-1}(N_j)=$ a $\tau_{2} -$open set. Hence, the mappings $F\circ f$ and $g\circ F$ are pairwise continuous maps from $(X,\tau_{1} ,\tau_{2})$ to $(Y,\psi_{1} ,\psi_{2})$.
\end{proof}

Now, we introduce $BDTS-$homotopy and $BDTS-$iteration homotopy.
\begin{definition}
Let $(X,f)$ and $(Y,g)$ be two bitopological dynamical systems $(BTDSs)$, where $f:(X,\tau_{1} ,\tau_{2})\rightarrow (X,\tau_{1} ,\tau_{2}) $ and $g:(Y,\psi_{1} ,\psi_{2})\rightarrow (Y,\psi_{1} ,\psi_{2}) $ are two pairwise continuous maps. Also, let $F:(X,\tau_{1} ,\tau_{2})\rightarrow (Y,\psi_{1} ,\psi_{2}) $  be a pairwise continuous map. Then, \\

(i) $f$ is said to be $BTDS-$homotopic  to $g$ relative to the map $F$ if there exists a pairwise continuous map $H:X\times [0,1] \rightarrow Y$ such that $H(x,0)=(F\circ f)(x)$ and $H(x,1)=(g\circ F)(x)$, for all $x \in X$.\\

The map $H$ is called a $BTDS-$homotopy between $f$ and $g$. If there exists a $BTDS-$homotopy between $f$ and $g$, we write $f\approx^{H}_{F} g$. Thus, geometrically $f\approx^{H}_{F} g$ simply means the pairwise continuous deformation of the map $F\circ f$ to the map $g\circ F$ as time $t$ varies from $0$ to $1$. Here, H in ``$\approx^{H}_{F}$" indicates the word `homotopy' in short. \\

(ii) $f$ is said to be $BTDS-$iteration homotopic to $g$ relative to the map $F$ if there exists a pairwise continuous map $H:X\times [0,1] \rightarrow Y$ such that $H(x_{n+1},0)=(F\circ f)(x_{n})$ and $H(x_{n+1},1)=(g\circ F)(x_{n})$, $n=0, 1,2,3,...$.\\

The map $H$ is called an $BTDS-$iteration homotopy between $f$ and $g$. If there exists an $BTDS-$iteration homotopy between $f$ and $g$, we write $f\approx^{IH}_{F} g$. Thus, geometrically $f\approx^{IH}_{F} g$ simply means the pairwise continuous deformation of the map $F\circ f$ to the map $g\circ F$ with respect to the iteration as time $t$ varies from $0$ to $1$.
\end{definition}

Now, we prove some fundamental results of $BTDS-$homotopy and\\ $BTDS-$iteration homotopy.
\begin{theorem}
$\approx^{H} $ is an equivalence relation.
\end{theorem}
\begin{proof} 
{\bf Reflexivity:} It is very easy to see that $f\approx^{H}_{F} f$; if we take $F=I,$ the identity map and so the map $H:X\times [0,1] \rightarrow X$ such that $H(x,0)=(F\circ f)(x)=(I\circ f)(x)=I(f(x))=f(x)$ and $H(x,1)=(f\circ F)(x)=(f\circ I)(x)=f(I(x))=f(x)$.\\
{\bf Symmetry:} Suppose that $f\approx^{H}_{F} g$ and that $H$ be a $BTDS-$homotopy between $f$ and $g$. Now, we define a map $H^* :X\times [0,1] \rightarrow Y$ such that $H^* (x,t) =H(x, 1-t)$. Then, $H^*$ is a pairwise continuous map because of the pairwise continuity of $H$. Also, $H^* (x,0)= H(x,1)=(g\circ F)(x)$ and $H^* (x,1)= H(x,0)=(F\circ f)(x)$. Thus, $H^* $ is a $BTDS-$homotopy between $g$ and $f$. Hence, $g\approx^{H}_{F} f$.\\
{\bf Transitivity:} Let $(X,f)$, $(Y,g)$ and $(Z,h)$ be three $BTDSs$, where $f:(X,\tau_{1} ,\tau_{2})\rightarrow (X,\tau_{1} ,\tau_{2}) $, $g:(Y,\psi_{1} ,\psi_{2})\rightarrow (Y,\psi_{1} ,\psi_{2}) $ and $h:(Z,\phi_{1} ,\phi_{2})\rightarrow (Z,\phi_{1} ,\phi_{2}) $ are three pairwise continuous maps. Also, let $F:X \rightarrow Y$ and $G:Y \rightarrow Z$ be pairwise continuous maps. Now, suppose that $f\approx^{H}_{F} g$ and $g\approx^{H}_{F} h$. We define the map $K:X \rightarrow Z$ as $K=G \circ F$. Then the map $K$, being the composition of two pairwise continuous maps, is pairwise continuous. Now, we define the map $H^{**}:X\times [0,1] \rightarrow Z$ such that $H^{**}(x,0)=(K\circ f)(x)$ and $H^{**}(x,1)=(h\circ K)(x)$. Clearly, $H^{**}$ is a pairwise continuous map. Thus, $H^{**} $ is a $BTDS-$homotopy between $f$ and $h$. Hence, $f\approx^{H}_{F} h$.

Thus, $\approx^{H} $ is an equivalence relation.
\end{proof}
\begin{theorem}
$\approx^{IH} $ is an equivalence relation.
\end{theorem}
We omit the proof of Theorem 3.2. as it can be proven similarly as the proof of Theorem 3.1. We provide following examples below.

\begin{example}
Let $ X=\{1, \frac{1}{2}, \frac{1}{3}, \frac{1}{4}\} $. We consider the bitopological dynamical system $(X,f)$, where $ (X,\tau_{1},\tau_{2})$ is a bitopological space with $\tau_1$ is the discrete topology and $\tau_2$ is the indiscrete topology. The pairwise continuous map $ f:X\rightarrow X $ is defined by $ f(x)=\frac{1}{2} $ $\forall x \in X$ . Further, let $ Y=\{1,\frac{1}{2},\frac{1}{3}\} $. Also, we consider the bitopological dynamical system $(Y,g)$, where $ (Y,\psi_{1},\psi_{2})$ is a bitopological space with $\psi_{1}=\{\phi,Y,\{1\},\{1,\frac{1}{2}\}\}$ and $\psi_{2}$ is the indiscrete topology. The pairwise continuous map $ g:Y\rightarrow Y $ is defined by $ g(1)=\frac{1}{2} $, $g(\frac{1}{2})=\frac{1}{3} $, $ g(\frac{1}{3})=1 $. Now, we define the map $F:(X,\tau_{1},\tau_{2}) \to (Y,\psi_{1},\psi_{2})$ by $ F(1)=1 $, $F(\frac{1}{2})=\frac{1}{2} $, $ F(\frac{1}{3})=\frac{1}{3} $, $ F(\frac{1}{4})=\frac{1}{3}$. Then, the map $F$ is pairwise continuous. Now, the map $(F \circ f):(X,\tau_{1},\tau_{2}) \to (Y,\psi_{1},\psi_{2})$ is defined by $ (F \circ f)(1)=\frac{1}{2} $, $(F \circ f)(\frac{1}{2})=\frac{1}{2} $, $( F \circ f)(\frac{1}{3})=\frac{1}{2} $, $ (F \circ f)(\frac{1}{4})=\frac{1}{2}$. Also, the map $(g \circ F):(X,\tau_{1},\tau_{2}) \to (Y,\psi_{1},\psi_{2})$ is defined by $ (g \circ F)(1)=\frac{1}{2} $, $(g \circ F)(\frac{1}{2})=\frac{1}{3} $, $ (g \circ F)(\frac{1}{3})=1 $, $( g \circ F)(\frac{1}{4})=1$. By lemma 3.1, the maps $F \circ f$ and $g \circ F$ are pairwise continuous. Now, we define the map $K:X\times [0,1] \rightarrow Y$ by $K(1,0)=\frac{1}{2}$, $K(\frac{1}{2},0)=\frac{1}{2}$, $K(\frac{1}{3},0)=\frac{1}{2}$, $K(\frac{1}{4},0)=\frac{1}{2}$, $K(1,t)=\frac{1}{2}$, $K(\frac{1}{2},t)=\frac{1}{3}$, $K(\frac{1}{3},t)=1$, $K(\frac{1}{4},t)=1$, where $0<t \leq 1$. Then, $K(x,0)=(F\circ f)(x)$ and $K(x,1)=(g\circ F)(x)$, for all $x \in X$. Also, the map $K$ is pairwise continuous. Hence, the map $K$ is a $BTDS-$homotopy between $f$ and $g$. Thus, $f\approx^{H}_{F} g$.
\end{example}
\begin{example} Let us consider the bitopological dynamical system $(\mathbb{R},f)$, where $ (\mathbb{R},\tau_{1},\tau_{2})$ is a bitopological space, $\tau_1$ is the usual  topology on $\mathbb{R}$ and $\tau_2$ is the discrete topology on $\mathbb{R}$. The pairwise continuous map $ f:\mathbb{R}\rightarrow \mathbb{R} $ is defined by $ f(x)=x+2 $. Also, we consider the bitopological dynamical system $(\mathbb{R},g)$, where $ (\mathbb{R},\psi_{1},\psi_{2})$ is the bitopological space as considered in Example 3.2. of \cite{2} and the pairwise continuous map $ g:\mathbb{R}\rightarrow \mathbb{R} $ is defined by $ g(x)=x+1 $. Now, we define the map $F:(\mathbb{R},\tau_{1},\tau_{2}) \to (\mathbb{R},\psi_{1},\psi_{2})$ by $F(x)=x$. Clearly, the map $F$ is pairwise continuous. Then, the maps $(F \circ f):(\mathbb{R},\tau_{1},\tau_{2}) \to (\mathbb{R},\psi_{1},\psi_{2})$ and $(g \circ F):(\mathbb{R},\tau_{1},\tau_{2}) \to (\mathbb{R},\psi_{1},\psi_{2})$ are given by $(F \circ f)=f$ and $(g \circ F)=g$, respectively, and so they are pairwise continuous. Now, we define the map $J:\mathbb{R}\times [0,1] \rightarrow \mathbb{R}$ by
\[ 
J(x,t)=\begin{cases} 
      x+2 & \text{when } \; t=0, \\
      x &  \text{when } \; 0<t<1, \\
      x+1 &  \text{when } \; t=1. 
   \end{cases}
\] 
Then, the map $J$ is pairwise continuous. Hence, the map $J$ is a $BTDS-$homotopy between $f$ and $g$. Thus, $f\approx^{H}_{F} g$.
\end{example}

Throughout this paper, we consider $I=[0,1]$ and $([0,1],\tau_1 ,\tau_2)$ as a bitopological space, where $\tau_1 = \tau_2=$ the usual topology on $[0,1]$. Using this, we introduce the concept of `path' in a bitopological space. 
\begin{definition}

Let $X$ be a bitopological space. If $f:[0,1]\to X$ is a pairwise continuous map such that $f(0)=x$ and $f(1)=y$, we call $f$ a bitopological path in $X$ from $x$ to $y$. Here, $x$ is the initial point and $y$ is the final point of the bitopological path $f$.
\end{definition}

Now, we introduce $BTDS-$path homotopy.
\begin{definition}
Let $([0,1],f)$ and $(X,g)$ be two bitopological dynamical systems $(BTDSs)$, where $f:[0,1] \rightarrow [0,1]$ and $g:(X,\psi_{1} ,\psi_{2})\rightarrow (X,\psi_{1} ,\psi_{2}) $ are two pairwise continuous maps. Also, let $F:[0,1] \rightarrow X$ be a bitopological path in $X$ from $x$ to $y$. Then, $f$ is said to be $BTDS-$path homotopic  to $g$ relative to the bitopological path $F$ if there exists a pairwise continuous map $H:[0,1]\times [0,1] \rightarrow X$ such that 
\begin{align*}
    H(m,0)=(F\circ f)(m) \; \; & \text{and} \; \; H(m,1)=(g\circ F)(m),\\
    H(0,n)=x \; \; & \text{and} \; \; H(1,n)=y.
\end{align*} 
for each $m \in (0,1)$, each $n \in [0,1]$ and $x,y \in X$.

The map $H$ is called a $BTDS-$path homotopy between $f$ and $g$. If there exists a $BTDS-$path homotopy between $f$ and $g$ from $x$ to $y$, we write $f\approx^{PH}_F g |x\to y$. Thus, $f\approx^{PH}_F g |x\to y$ simply means the pairwise continuous deformation of the bitopological path $F\circ f$ to the bitopological path $g\circ F$ as time $t$ varies from $0$ to $1$.
\end{definition}

\begin{theorem}
$\approx^{PH} $ is an equivalence relation.
\end{theorem}
We omit the proof of Theorem 3.3. as it can be proven similarly as of the proof of Theorem 3.1.

\section{Homotopy based selection principles in Bitopological dynamical systems}

In this section, we introduce $BTDS$ homotopy based selection principles and discuss their various properties. First, we introduce $BTDS$ homotopy based selection hypothesis.
\begin{definition}
Let $(X,f)$ and $(Y,g)$ be two $BTDSs$. If $\mathcal{A}$ and $\mathcal{B}$ are sets formed by the families of subsets
of $Y$ and  $f\approx^{H}_F g$, then:
\begin{enumerate}
    \item {\bf  $S_1^{H}(\mathcal{A},\mathcal{B})$ is a BTDS homotopy based selection hypothesis:} for each sequence $\mathcal{X}=(\mathcal{U}_n: n\in \mathbb{N})$ of elements of $\mathcal{A}$ and for any point $x \in X$ such that $H(x,0)\in \bigcup \mathcal{U}_n$, for each $n \in \mathbb{N}$; there is a sequence $(U_n : n \in \mathbb{N})$ such that $U_n \in \mathcal{U}_n$, for each $n \in \mathbb{N}$ and $H(x, 1)\in U_n$, for each $n \in \mathbb{N}$ and $\{U_n : n \in \mathbb{N}\} \in \mathcal{B}$.
    \item {\bf  $S_{fin}^{H}(\mathcal{A},\mathcal{B})$ is a BTDS  homotopy based selection hypothesis:} for each sequence $\mathcal{X}=(\mathcal{U}_n: n\in \mathbb{N})$ of elements of $\mathcal{A}$ and for any point $x \in X$ such that $H(x,0)\in \bigcup \mathcal{U}_n$, for each $n \in \mathbb{N}$; there is a sequence $\mathcal{Y}=(\mathcal{V}_n: n\in \mathbb{N})$ such that for each $n\in \mathbb{N}$, $\mathcal{V}_n$ is a finite subset of $\mathcal{U}_n$ and $H(x,1)\in \bigcup \mathcal{V}_n$, for each $n \in \mathbb{N}$ and $\bigcup\atop_{n \in \mathbb{N}}$ $\mathcal{V}_n\in \mathcal{B}$.
\end{enumerate}
\end{definition}
\begin{definition}
For a bitopological space $(Y,\psi_{1} ,\psi_{2})$, we use the following notations with the convention that $i,j \in \{1,2\}$, $i\neq j$:

\begin{enumerate}
    \item $\mathcal{O}_{i}$ is the collection of all $\psi_{i}-$open covers of $Y$,
    \item $\mathcal{O}^D_{ij}=\{ \mathcal{U} \subseteq \psi_{i}: \psi_{j}-cl(\bigcup \mathcal{U})=Y \} $,
    \item $\overline{\mathcal{O}_{ij}}=\{ \mathcal{U} \subseteq \psi_{i}: \bigcup \{\psi_{j}-cl(U): U \in  \mathcal{U}\} =Y \}$.
\end{enumerate}
\end{definition}
\begin{definition}
Let $(X,f)$ and $(Y,g)$ be two $BTDSs$, where $(X,\tau_{1} ,\tau_{2})$ and $(Y,\psi_{1} ,\psi_{2})$ are two bitopological spaces. If  $f\approx^{H}_F g$, and $i,j \in \{1,2\}$, $i\neq j$, then:

\begin{enumerate}
    \item $S_1^{H}(\mathcal{O}_{i},\mathcal{O}_{j})$ is called  $(i,j)$-H-Rothberger property of $Y$. If $Y$ has both $(1,2)$-H-Rothberger property and $(2,1)$-H-Rothberger property, then $Y$ is said to have H-Rothberger property.
    \item $S_1^{H}(\mathcal{O}_{i},\overline{\mathcal{O}_{ij}})$ is called  $(i,j)$-H-almost Rothberger property of $Y$. If $Y$ has both $(1,2)$-H-almost Rothberger  property and $(2,1)$-H-almost Rothberger property, then $Y$ is said to have H-almost Rothberger property.
    \item $S_1^{H}(\mathcal{O}_{i},\mathcal{O}^D_{ij})$ is called  $(i,j)$-H-weak Rothberger property of $Y$. If $Y$ has both $(1,2)$-H-weak Rothberger  property and $(2,1)$-H-weak Rothberger property, then $Y$ is said to have H-weak Rothberger property.
    \item $S_{fin}^{H}(\mathcal{O}_{i},\mathcal{O}_{j})$ is called  $(i,j)$-H-Menger property  of $Y$. If $Y$ has both $(1,2)$-H-Menger  property and $(2,1)$-H-Menger  property, then $Y$ is said to have H-Menger  property.
    \item  $S_{fin}^{H}(\mathcal{O}_{i},\overline{\mathcal{O}_{ij}})$ is called  $(i,j)$-H-almost Menger property of $Y$. If $Y$ has both $(1,2)$-H-almost Menger  property and $(2,1)$-H-almost Menger property, then $Y$ is said to have H-almost Menger property.
    \item $S_{fin}^{H}(\mathcal{O}_{i},\mathcal{O}^D_{ij})$ is called  $(i,j)$-H-weak Menger property of $Y$. If $Y$ has both $(1,2)$-H-weak Menger  property and $(2,1)$-H-weak Menger property, then $Y$ is said to have H-weak Menger property.
\end{enumerate}
\end{definition}

To preserve the sense of iteration of Definition 3.1, we will write $HI$ in lieu of $H$ in the above definition. It will not alter any meaning except indicating $H(x_{n+1}, 0)= (F\circ f)(x_n)$ and $H(x_{n+1}, 1)=(g\circ F)(x_n)$, where $n=0,1,2,...$.

Now, we introduce $BTDS$ path homotopy based selection hypothesis.
\begin{definition}
Let $([0,1],f)$ and $(X,g)$ be two $BTDSs$. If $\mathcal{A}$ and $\mathcal{B}$ are sets formed by the families of subsets
of $X$ and  $f\approx^{PH}_F g |x\to y$, then:

\begin{enumerate}
    \item {\bf  $S_1^{PH}(\mathcal{A},\mathcal{B})$ is a BTDS path homotopy based selection hypothesis:} for each sequence $\mathcal{X}=(\mathcal{U}_n: n\in \mathbb{N})$ of elements of $\mathcal{A}$ and for any point $m \in (0,1)$ such that $H(m,0)\in \bigcup \mathcal{U}_n$, for each $n \in \mathbb{N}$; there is a sequence $(U_n : n \in \mathbb{N})$ such that $U_n \in \mathcal{U}_n$, for each $n \in \mathbb{N}$ and $H(m, 1)\in U_n$, for each $n \in \mathbb{N}$ and $\{U_n : n \in \mathbb{N}\} \in \mathcal{B}$.
    \item {\bf  $S_{fin}^{PH}(\mathcal{A},\mathcal{B})$ is a BTDS  path homotopy based selection hypothesis:} for each sequence $\mathcal{X}=(\mathcal{U}_n: n\in \mathbb{N})$ of elements of $\mathcal{A}$ and for any point $m \in (0,1)$ such that $H(m,0)\in \bigcup \mathcal{U}_n$, for each $n \in \mathbb{N}$; there is a sequence $\mathcal{Y}=(\mathcal{V}_n: n\in \mathbb{N})$ such that for each $n\in \mathbb{N}$, $\mathcal{V}_n$ is a finite subset of $\mathcal{U}_n$ and $H(m,1)\in \bigcup \mathcal{V}_n$, for each $n \in \mathbb{N}$ and $\bigcup\atop_{n \in \mathbb{N}}$ $\mathcal{V}_n\in \mathcal{B}$.
\end{enumerate}
\end{definition}
\begin{definition}
Let $([0,1],f)$ and $(X,g)$ be two $BTDSs$, where $([0,1],\tau_{1} ,\tau_{2})$ and $(X,\psi_{1} ,\psi_{2})$ are two bitopological spaces. If  $f\approx^{PH}_F g|x\to y$, then:

\begin{enumerate}
    \item $S_1^{PH}(\mathcal{O}_{i},\mathcal{O}_{j})$ is called  $(i,j)$-PH-Rothberger property of $X$. If $X$ has both $(1,2)$-PH-Rothberger property and $(2,1)$-PH-Rothberger property, then $X$ is said to have PH-Rothberger property.
    \item $S_1^{PH}(\mathcal{O}_{i},\overline{\mathcal{O}_{ij}})$ is called  $(i,j)$-PH-almost Rothberger property of $X$. If $X$ has both $(1,2)$-PH-almost Rothberger  property and $(2,1)$-PH-almost Rothberger property, then $X$ is said to have PH-almost Rothberger property.
    \item $S_1^{PH}(\mathcal{O}_{i},\mathcal{O}^D_{ij})$ is called  $(i,j)$-PH-weak Rothberger property of $X$. If $X$ has both $(1,2)$-PH-weak Rothberger  property and $(2,1)$-PH-weak Rothberger property, then $X$ is said to have PH-weak Rothberger property.
    \item $S_{fin}^{PH}(\mathcal{O}_{i},\mathcal{O}_{j})$ is called  $(i,j)$-PH-Menger property  of $X$. If $X$ has both $(1,2)$-PH-Menger  property and $(2,1)$-PH-Menger  property, then $X$ is said to have PH-Menger  property.
    \item $S_{fin}^{PH}(\mathcal{O}_{i},\overline{\mathcal{O}_{ij}})$ is called  $(i,j)$-PH-almost Menger property of $X$. If $X$ has both $(1,2)$-PH-almost Menger  property and $(2,1)$-PH-almost Menger property, then $X$ is said to have PH-almost Menger property.
    \item $S_{fin}^{PH}(\mathcal{O}_{i},\mathcal{O}^D_{ij})$ is called  $(i,j)$-PH-weak Menger property of $X$. If $X$ has both $(1,2)$-PH-weak Menger  property and $(2,1)$-PH-Menger  property, then $X$ is said to have PH-weak Menger property.
\end{enumerate}
\end{definition}
The following theorem establishes the relationships between various Menger properties.
\begin{theorem} Let $(X,f)$ and $(Y,g)$ be two $BTDSs$, where $(X,\tau_{1} ,\tau_{2})$ and $(Y,\psi_{1} ,\psi_{2})$ are two bitopological spaces. If  $f\approx^{H}_F g$, $i,j \in \{1,2\}$ and $i\neq j$, then
$Y$ has H-Menger property $\implies$ $Y$ has H-almost Menger property $ \implies$ $Y$ has H-weak Menger property.
\end{theorem}
\begin{proof} 
{\bf (i)} Suppose that $Y$ has H-Menger property. Then, $Y$ has both $(1,2)$-H-Menger property and $(2,1)$-H-Menger property. We shall prove that if $Y$ has $(1,2)$-H-Menger property, then $Y$ has $(1,2)$-H-almost Menger  property. The proof of $(2,1)$-H-Menger property implies $(2,1)$-H-almost Menger  property will be similar.

Let $Y$ has $(1,2)$-H-Menger property. Suppose that $\mathcal{X}=(\mathcal{U}_n: n\in \mathbb{N})$ is an arbitrary sequence of $\psi_{1}$-open covers of $Y$ and that $x \in X$ is an arbitrary point such that $H(x,0)\in \bigcup \mathcal{U}_n$, for each $n \in \mathbb{N}$. Then, there is a sequence $\mathcal{Y}=(\mathcal{V}_n: n\in \mathbb{N})$ such that for each $n\in \mathbb{N}$, $\mathcal{V}_n$ is a finite subset of $\mathcal{U}_n$ and $H(x,1)\in \bigcup \mathcal{V}_n$, $\mathcal{V}_n\in \mathcal{Y}$, for each $n \in \mathbb{N}$ and $\bigcup\atop_{n \in \mathbb{N}}$ $\mathcal{V}_n\in \mathcal{O}_{2}$. This gives $Y=\bigcup_{n \in \mathbb{N}} \mathcal{V}_n$.

Now for any $V \in \mathcal{V}_n$, for some $n$, $V \subseteq \psi_{2}-cl(V)$. This implies $\mathcal{V}_n \subseteq \bigcup_{V \in\mathcal{V}_n} \psi_{2}-cl(V)$. This gives $ \bigcup_{n \in \mathbb{N}} \mathcal{V}_n \subseteq \bigcup_{n \in \mathbb{N}} (\bigcup_{V \in\mathcal{V}_n} \psi_{2}-cl(V))$. This implies $ Y \subseteq \bigcup_{n \in \mathbb{N}} (\bigcup_{V \in\mathcal{V}_n} \psi_{2}-cl(V))$. Also, $ \bigcup_{n \in \mathbb{N}} (\bigcup_{V \in\mathcal{V}_n} \psi_{2}-cl(V)) \subseteq Y$. Thus, $ Y = \bigcup_{n \in \mathbb{N}} (\bigcup_{V \in\mathcal{V}_n} \psi_{2}-cl(V))$. This gives $\bigcup_{n \in \mathbb{N}} \mathcal{V}_n \in \overline{\mathcal{O}_{12}}$. Hence, $Y$ has $(1,2)$-H-almost Menger property.\\

{\bf (ii)} Suppose that $Y$ has H-almost Menger property. Then, $Y$ has both $(1,2)$-H-almost Menger property and $(2,1)$-H-almost Menger property. We shall prove that if $Y$ has $(1,2)$-H-almost Menger property, then $Y$ has $(1,2)$-H-weak Menger  property. The proof of $(2,1)$-H-almost Menger property implies $(2,1)$-H-weak Menger  property will be similar.

Let $Y$ has $(1,2)$-H-almost Menger property. Suppose that $\mathcal{X}=(\mathcal{U}_n: n\in \mathbb{N})$ is an arbitrary sequence of $\psi_{1}$-open covers of $Y$ and that $x \in X$ is an arbitrary point such that $H(x,0)\in \bigcup \mathcal{U}_n$, for each $n \in \mathbb{N}$. Then, there is a sequence $\mathcal{Y}=(\mathcal{V}_n: n\in \mathbb{N})$ such that for each $n\in \mathbb{N}$, $\mathcal{V}_n$ is a finite subset of $\mathcal{U}_n$ and $H(x,1)\in \bigcup \mathcal{V}_n$, $\mathcal{V}_n\in \mathcal{Y}$, for each $n \in \mathbb{N}$ and $\bigcup\atop_{n \in \mathbb{N}}$ $\mathcal{V}_n\in \overline{\mathcal{O}_{12}}$. This gives $Y=\bigcup_{n \in \mathbb{N}} (\bigcup_{V \in \mathcal{V}_n} \psi_{2}-cl(V))$. 

Now for some fixed $n$, $\psi_{2}-cl(\bigcup_{V \in \mathcal{V}_n} V) \supseteq \bigcup_{V \in \mathcal{V}_n} \psi_{2}-cl(V) $. This gives $ \bigcup_{n \in \mathbb{N}} \psi_{2}-cl(\bigcup_{V \in \mathcal{V}_n} V) \supseteq \bigcup_{n \in \mathbb{N}} (\bigcup_{V \in \mathcal{V}_n} \psi_{2}-cl(V)) $. Also, $\psi_{2}-cl(\bigcup_{n\in \mathbb{N}} \mathcal{V}_n) \supseteq \bigcup_{n \in \mathbb{N}} \psi_{2}-cl(\bigcup_{V \in \mathcal{V}_n} V)$. This gives $\psi_{2}-cl(\bigcup_{n\in \mathbb{N}} \mathcal{V}_n) \supseteq \bigcup_{n \in \mathbb{N}} (\bigcup_{V \in \mathcal{V}_n} \psi_{2}-cl(V))$. This implies $Y \subseteq \psi_{2}-cl(\bigcup_{n\in \mathbb{N}} \mathcal{V}_n)$. Also, $\psi_{2}-cl(\bigcup_{n\in \mathbb{N}} \mathcal{V}_n) \subseteq Y$. This gives $Y = \psi_{2}-cl(\bigcup_{n\in \mathbb{N}} \mathcal{V}_n)$. Thus, $\bigcup_{n \in \mathbb{N}} \mathcal{V}_n \in \mathcal{O}^D_{12}$. Hence, $Y$ has $(1,2)$-H-weak Menger property.

{\bf (iii)} This part can be proved by combining part (i) and part (ii).
\end{proof}
The following theorem establishes the relationships between various Rothberger properties.
\begin{theorem} Let $(X,f)$ and $(Y,g)$ be two $BTDSs$, where $(X,\tau_{1} ,\tau_{2})$ and $(Y,\psi_{1} ,\psi_{2})$ are two bitopological spaces. If  $f\approx^{H}_F g$, $i,j \in \{1,2\}$ and $i\neq j$, then $Y$ has H-Rothberger property $\implies$ $Y$ has H-almost Rothberger property $\implies$ $Y$ has H-weak Rothberger property.
\end{theorem}
The following theorem establishes the relationships between corresponding Rothberger and Menger properties.
\begin{theorem} Let $(X,f)$ and $(Y,g)$ be two $BTDSs$, where $(X,\tau_{1} ,\tau_{2})$ and $(Y,\psi_{1} ,\psi_{2})$ are two bitopological spaces. If  $f\approx^{H}_F g$, $i,j \in \{1,2\}$ and $i\neq j$ then:
\begin{enumerate}
    \item[(i)] $Y$ has H-Rothberger property $\implies$ $Y$ has H-Menger property,
    \item[(ii)] $Y$ has H-almost Rothberger property $ \implies$ $Y$ has H-almost Menger property,
    \item[(iii)] $Y$ has H-weak Rothberger property $\implies$ $Y$ has H-weak Menger property.
\end{enumerate}
\end{theorem}

In view of the above theorems, we have the following diagram:

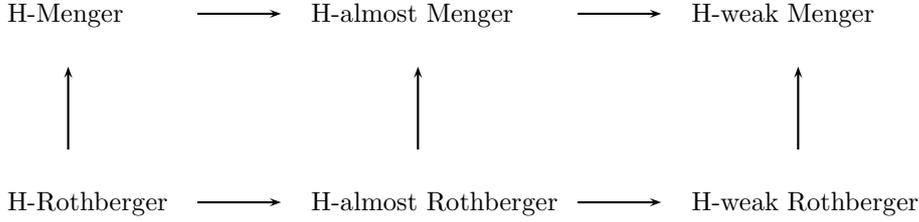
\begin{pspicture}(12,4)
\put(0,0){H-Rothberger}
\put(4,0){H-almost Rothberger}
\put(9,0){H-weak Rothberger}
\put(0,2.5){H-Menger}
\put(4,2.5){H-almost Menger}
\put(9,2.5){H-weak Menger}
\psline{->}(0.8,0.8)(0.8,1.9)
\psline{->}(2.5,0.1)(3.6,0.1)
\psline{->}(5.4,0.8)(5.4,1.9)
\psline{->}(7.5,0.1)(8.6,0.1)
\psline{->}(10.4,0.8)(10.4,1.9)
\psline{->}(2.5,2.6)(3.6,2.6)
\psline{->}(7.5,2.6)(8.6,2.6)
\end{pspicture}\\

The following theorems establish the relationships between various Rothberger properties, various Menger properties and corresponding Rothberger and Menger properties in path homotopy. We skip the proofs of the following three theorems.
\begin{theorem} Let $([0,1],f)$ and $(X,g)$ be two $BTDSs$, where $([0,1],\tau_{1} ,\tau_{2})$ and $(X,\psi_{1} ,\psi_{2})$ are two bitopological spaces. If  $f\approx^{PH}_F g|x\to y$, $i,j \in \{1,2\}$ and $i\neq j$, then
$X$ has PH-Rothberger property $\implies$ $X$ has PH-almost Rothberger property $ \implies$ $X$ has PH-weak Rothberger property.
\end{theorem}
\begin{theorem} Let $([0,1],f)$ and $(X,g)$ be two $BTDSs$, where $([0,1],\tau_{1} ,\tau_{2})$ and $(X,\psi_{1} ,\psi_{2})$ are two bitopological spaces. If  $f\approx^{PH}_F g|x\to y$, $i,j \in \{1,2\}$ and $i\neq j$, then
$X$ has PH-Menger property $\implies$ $X$ has PH-almost Menger property $ \implies$ $X$ has PH-weak Menger property.
\end{theorem}
\begin{theorem} Let $([0,1],f)$ and $(X,g)$ be two $BTDSs$, where $([0,1],\tau_{1} ,\tau_{2})$ and $(X,\psi_{1} ,\psi_{2})$ are two bitopological spaces. If  $f\approx^{PH}_F g|x\to y$, $i,j \in \{1,2\}$ and $i\neq j$, then
\begin{enumerate}
    \item[(i)] $X$ has PH-Rothberger property $\implies$ $X$ has PH-Menger property,
    \item[(ii)] $X$ has PH-almost Rothberger property $ \implies$ $X$ has PH-almost Menger property,
    \item[(iii)] $X$ has PH-weak Rothberger property $\implies$ $X$ has PH-weak Menger property.
\end{enumerate}
\end{theorem}

Converse of the above theorems 4.1, 4.2, 4.3, 4.4, 4.5 and 4.6 are not true in general. We give the following counterexample regarding the converse of Theorem 4.3. part (ii). Other counterexamples can be found easily.

\begin{example} Let us consider the bitopological dynamical system $(\mathbb{R},f)$, where $ (\mathbb{R},\tau_{1},\tau_{2})$ is a bitopological space, $\tau_1$ is the lower limit topology on $\mathbb{R}$ and $\tau_2$ is the discrete topology on $\mathbb{R}$. The pairwise continuous map $ f:\mathbb{R}\rightarrow \mathbb{R} $ is defined by $ f(x)=x+2 $. Also, we consider the bitopological dynamical system $(\mathbb{R},g)$; where $ (\mathbb{R},\psi_{1},\psi_{2})$ is a bitopological space, $\psi_{1}$ is the usual topology on $\mathbb{R}$ and $\psi_{2}$ is the discrete topology on $\mathbb{R}$. The pairwise continuous map $ g:\mathbb{R}\rightarrow \mathbb{R} $ is defined by $ g(x)=x+1 $. Now, we define the map $F:(\mathbb{R},\tau_{1},\tau_{2}) \to (\mathbb{R},\psi_{1},\psi_{2})$ by $F(x)=x$. Clearly, the map $F$ is pairwise continuous. It is easy to check that the map $H:\mathbb{R}\times [0,1] \rightarrow \mathbb{R}$ defined by
\[ 
H(x,t)=\begin{cases} 
      x+2 & \text{when} \; t=0, \\
      x & \text{when} \; 0<t<1, \\
      x+1  & \text{when} \; t=1. 
   \end{cases}
\] 
is a $BTDS-$homotopy between $f$ and $g$. 

Now, let $\mathcal{X}=(\mathcal{U}_n: n\in \mathbb{N})$ be a sequence of $\psi_{1}$-open covers of $\mathbb{R}$ such that for any $n\in \mathbb{N}$ and $U \in \mathcal{U}_n$, we have $D(U)<\frac{1}{3^n}$, where $D(U)$ is the diameter of the set $U$ i.e. $D(U)=\sup \{|x-y|:x,y \in U\}$. It is easy to check that for any arbitrary point $x \in X$, $H(x,0)\in \bigcup \mathcal{U}_n$, for each $n \in \mathbb{N}$ as $\mathcal{U}_n$ covers $\mathbb{R}$. For simplicity, we consider $x=0$. Now, for any sequence $(U_n : n \in \mathbb{N})$ such that $U_n \in \mathcal{U}_n$, for each $n \in \mathbb{N}$ and $H(0, 1)\in U_n$, for each $n \in \mathbb{N}$; we have $U_n =(1-\epsilon,1+\epsilon)$, where $\epsilon<\frac{1}{3}$ depends on $n$. This gives $\psi_{2}-cl(U_n)=U_n$. Thus, $ \mathbb{R}\neq \bigcup_{n \in \mathbb{N}} \psi_{2}-cl(U_n)$. Hence, $\mathbb{R}$ doesn't have  $(1,2)$-H-almost Rothberger property. 

We show that $\mathbb{R}$ has $(1,2)$-H-almost Menger property. Suppose that $\mathcal{X}=(\mathcal{U}_n: n\in \mathbb{N})$ is an arbitrary sequence of $\psi_{1}$-open covers of $\mathbb{R}$ and that $x \in X$ is an arbitrary point such that $H(x,0)\in \bigcup \mathcal{U}_n$, for each $n \in \mathbb{N}$. We can construct a sequence $\mathcal{Y}=(\mathcal{V}_n: n\in \mathbb{N})$ such that for each $n\in \mathbb{N}$, $\mathcal{V}_n$ is a finite subset of $\mathcal{U}_n$ by particularly choosing the set of $\mathcal{U}_n$ containing $H(x,1)$ in $\mathcal{V}_n$, for each $n$ so that $H(x,1)\in \bigcup \mathcal{V}_n$, $\mathcal{V}_n\in \mathcal{Y}$, for each $n \in \mathbb{N}$ (it is important to note that $H(x,1) \in \bigcup \mathcal{U}_n$ as $\mathcal{U}_n$ covers $\mathbb{R}$) and also choosing other sets of $\mathcal{V}_n$ so that $Y=\bigcup_{n \in \mathbb{N}} \psi_{2}-cl(U_n)$.
\end{example}

Although converse of the theorems 4.1, 4.2, 4.3, 4.4, 4.5 and 4.6 are not true in general; but, if the bitopological space $(Y,\psi_{1} ,\psi_{2})$ has some additional property, then we have the equalities. For this purpose, we introduce the pairwise $P$-space in bitopological space.
\begin{definition}
A bitopological space $ (X,\tau_{1} ,\tau_{2}) $ is  called a pairwise $P$-space if both $ (X,\tau_{1})$ and $ (X,\tau_{2})$ are $P$-spaces.
\end{definition}

\begin{theorem}
Let $(X,f)$ and $(Y,g)$ be two $BTDSs$, where $(X,\tau_{1} ,\tau_{2})$ and $(Y,\psi_{1} ,\psi_{2})$ are two bitopological spaces. If  $f\approx^{H}_F g$ and if $(Y,\psi_{1} ,\psi_{2})$ is a pairwise $T_3$ and pairwise locally compact space , then
    $Y$ has H-Rothberger property iff $Y$ has H-almost Rothberger property.
\end{theorem}
\begin{proof}
The necessary part follows from Theorem 4.2. For the sufficient part, suppose that $Y$ has H-almost Rothberger property. Then, $Y$ has both $(1,2)$-H-almost Rothberger property and $(2,1)$-H-almost Rothberger property. We shall prove that if $Y$ has $(1,2)$-H-almost Rothberger property, then $Y$ has $(1,2)$-H-Rothberger  property. The proof of $(2,1)$-H-almost Rothberger property implies $(2,1)$-H-Rothberger  property will be similar.

Let $Y$ has $(1,2)$-H-almost Rothberger property. Suppose that $\mathcal{X}=(\mathcal{U}_n: n\in \mathbb{N})$ is an arbitrary sequence of $\psi_{1}$-open covers of $Y$ and that $x \in X$ is an arbitrary point such that $H(x,0)\in \bigcup \mathcal{U}_n$, for each $n \in \mathbb{N}$. Since $Y$ is pairwise $T_3$, so it is pairwise regular. Hence, $\psi_{1}$ is regular with respect to $\psi_{2}$. This implies that for each $n\in \mathbb{N}$, there exists a $\psi_{1}$-open cover $\mathcal{V}_n$ of $Y$ such that $\mathcal{V}_{n}^{*}=\{\psi_{2}-cl(V):V \in \mathcal{V}_n\}$ is a refinement of $\mathcal{U}_n$, where $H(x,0)\in \bigcup \mathcal{V}_n$,  for each $n \in \mathbb{N}$. Now, $(1,2)$-H-almost Rothberger property of $Y$ implies that there is a sequence $(V_n : n \in \mathbb{N})$ such that $V_n \in \mathcal{V}_n$, for each $n \in \mathbb{N}$ and $H(x, 1)\in V_n$, for each $n \in \mathbb{N}$ and $\{V_n : n \in \mathbb{N}\} \in \overline{\mathcal{O}_{12}}$. This gives $Y=\bigcup_{n \in \mathbb{N}} \psi_{2}-cl(V_n)$. Since for each $n \in \mathbb{N}$, $\mathcal{V}_{n}^{*}$ is a refinement of $\mathcal{U}_n$; so for each $n \in \mathbb{N}$, there exists $U_n \in \mathcal{U}_n$ such that $\psi_{2}-cl(V_n) \subseteq U_n$. This gives $\bigcup_{n \in \mathbb{N}} \psi_{2}-cl(V_n) \subseteq \bigcup_{n \in \mathbb{N}} U_n$. Hence, $Y \subseteq \bigcup_{n \in \mathbb{N}} U_n$. Also for each $n \in \mathbb{N}$, $H(x, 1)\in V_n \subseteq \psi_{2}-cl(V_n) \subseteq U_n$. 

Now, since $Y$ is pairwise $T_3$ and pairwise locally compact, so it is pairwise regular, pairwise Hausdorff and pairwise locally compact. Then by proposition 2.2, $\psi_{1} \subseteq \psi_{2}$. This implies that $\{U_n : n \in \mathbb{N}\} \in \mathcal{O}_{2}$. Thus, $(Y,\psi_{1} ,\psi_{2})$ is $(1,2)$-H-Rothberger.
\end{proof}
\begin{theorem}
Let $(X,f)$ and $(Y,g)$ be two $BTDSs$, where $(X,\tau_{1} ,\tau_{2})$ and $(Y,\psi_{1} ,\psi_{2})$ are two bitopological spaces. If  $f\approx^{H}_F g$ and if $(Y,\psi_{1} ,\psi_{2})$ is a pairwise $P$-space, then
    $Y$ has H-almost Rothberger property iff $Y$ has H-weak Rothberger property.
\end{theorem}
\begin{proof}
The necessary part follows from theorem 4.2. For the sufficient part, suppose that $Y$ has H-weak Rothberger property. Then, $Y$ has both $(1,2)$-H-weak Rothberger property and $(2,1)$-H-weak Rothberger property. We shall prove that if $Y$ has $(1,2)$-H-weak Rothberger property, then $Y$ has $(1,2)$-H-almost Rothberger  property. The proof of $(2,1)$-H-weak Rothberger property implies $(2,1)$-H-almost Rothberger  property will be similar.

Suppose that $\mathcal{X}=(\mathcal{U}_n: n\in \mathbb{N})$ is an arbitrary sequence of $\psi_{1}$-open covers of $Y$ and that $x \in X$ is an arbitrary point such that $H(x,0)\in \bigcup \mathcal{U}_n$, for each $n \in \mathbb{N}$. The $(1,2)$-H-weak Rothberger property of $Y$ implies that there is a sequence $(U_n : n \in \mathbb{N})$ such that $U_n \in \mathcal{U}_n$, for each $n \in \mathbb{N}$ and $H(x, 1)\in U_n$, for each $n \in \mathbb{N}$ and $\{U_n : n \in \mathbb{N}\} \in \mathcal{O}^D_{12}$. This gives $Y = \psi_{2}-cl(\bigcup_{n \in \mathbb{N}} U_n)$. 

Now for each $n \in \mathbb{N}$, $U_n \subseteq \psi_{2}-cl(U_n)$. This gives $\bigcup_{n \in \mathbb{N}}U_n \subseteq \bigcup_{n \in \mathbb{N}} \psi_{2}-cl(U_n)$. Since $Y$ is a pairwise $P$-space, so $\bigcup_{n \in \mathbb{N}} \psi_{2}-cl(U_n)$ is a $\psi_{2}$-closed set containing $\bigcup_{n \in \mathbb{N}}U_n$. Also, $\psi_{2}-cl(\bigcup_{n \in \mathbb{N}} U_n)$ is the smallest $\psi_{2}$-closed set containing $\bigcup_{n \in \mathbb{N}}U_n$. This gives $\psi_{2}-cl(\bigcup_{n \in \mathbb{N}} U_n) \subseteq \bigcup_{n \in \mathbb{N}} \psi_{2}-cl(U_n)$. Thus, $Y \subseteq \bigcup_{n \in \mathbb{N}} \psi_{2}-cl(U_n)$. Also, $\bigcup_{n \in \mathbb{N}} \psi_{2}-cl(U_n) \subseteq Y$. Hence, $Y = \bigcup_{n \in \mathbb{N}} \psi_{2}-cl(U_n)$. This gives $\{U_n : n \in \mathbb{N}\} \in \overline{\mathcal{O}_{12}}$. Thus, $Y$ is $(1,2)$-H-almost Rothberger. 
\end{proof}
Theorem 4.7. and theorem 4.8. together gives the following corollary.
\begin{corollary}
Let $(X,f)$ and $(Y,g)$ be two $BTDSs$, where $(X,\tau_{1} ,\tau_{2})$ and $(Y,\psi_{1} ,\psi_{2})$ are two bitopological spaces. If  $f\approx^{H}_F g$ and if $(Y,\psi_{1} ,\psi_{2})$ is a pairwise $T_3$ and pairwise locally compact pairwise $P$-space, then
    $Y$ has H-Rothberger property iff $Y$ has H-weak Rothberger property.
\end{corollary}
\begin{theorem} Let $(X,f)$ and $(Y,g)$ be two $BTDSs$, where $(X,\tau_{1} ,\tau_{2})$ and $(Y,\psi_{1} ,\psi_{2})$ are two bitopological spaces. If  $f\approx^{H}_F g$ and if $(Y,\psi_{1} ,\psi_{2})$ is a pairwise $T_3$ and pairwise locally compact pairwise $P$-space, then following conditions are equivalent:
\begin{enumerate}
    \item[(i)] $Y$ has H-Rothberger property,
    \item[(ii)] $Y$ has H-almost Rothberger property,
    \item[(iii)] $Y$ has H-weak Rothberger property.
\end{enumerate}
\end{theorem}
\begin{proof}
Proof follows by combining Theorem 4.7. and Theorem 4.8.
\end{proof}
\begin{theorem}
Let $(X,f)$ and $(Y,g)$ be two $BTDSs$, where $(X,\tau_{1} ,\tau_{2})$ and $(Y,\psi_{1} ,\psi_{2})$ are two bitopological spaces. If  $f\approx^{H}_F g$ and if $(Y,\psi_{1} ,\psi_{2})$ is a pairwise $T_3$ and pairwise locally compact space , then $Y$ has H-Menger property iff $Y$ has H-almost Menger property.
\end{theorem}
\begin{proof}
The necessary part follows from Theorem 4.1. For the sufficient part, suppose that $Y$ has H-almost Menger property. Then, $Y$ has both $(1,2)$-H-almost Menger property and $(2,1)$-H-almost Menger property. We shall prove that if $Y$ has $(1,2)$-H-almost Menger property, then $Y$ has $(1,2)$-H-Menger  property. The proof of $(2,1)$-H-almost Menger property implies $(2,1)$-H-Menger  property will be similar.

Suppose that $\mathcal{X}=(\mathcal{U}_n: n\in \mathbb{N})$ is an arbitrary sequence of $\psi_{1}$-open covers of $Y$ and that $x \in X$ is an arbitrary point such that $H(x,0)\in \bigcup \mathcal{U}_n$, for each $n \in \mathbb{N}$. Since $Y$ is pairwise $T_3$, so it is pairwise regular i.e. $\psi_{1}$ is regular with respect to $\psi_{2}$. This implies that for each $n\in \mathbb{N}$, there exists a $\psi_{1}$-open cover $\mathcal{U}^{/}_n$ of $Y$ such that $\mathcal{U}^{*}_n=\{\psi_{2}-cl(U^{/}):U^{/} \in \mathcal{U}^{/}_n\}$ is a refinement of $\mathcal{U}_n$, where $H(x,0)\in \bigcup \mathcal{U}^{/}_n$,  for each $n \in \mathbb{N}$. Now, $(1,2)$-H-almost Menger property of $Y$ implies that there is a sequence $\mathcal{Y}^/=(\mathcal{V}^{/}_n: n\in \mathbb{N})$ such that for each $n\in \mathbb{N}$, $\mathcal{V}^{/}_n$ is a finite subset of $\mathcal{U}^{/}_n$ and $H(x,1)\in \bigcup \mathcal{V}^{/}_n$, $\mathcal{V}^{/}_n\in \mathcal{Y}^{/}$, for each $n \in \mathbb{N}$ and $\bigcup\atop_{n \in \mathbb{N}}$ $\mathcal{V}^{/}_n\in \overline{\mathcal{O}_{12}}$. This gives $Y=\bigcup_{n \in \mathbb{N}} (\bigcup_{V \in \mathcal{V}^{/}_n} \psi_{2}-cl(V))$. Since for each $n \in \mathbb{N}$, $\mathcal{U}_{n}^{*}$ is a refinement of $\mathcal{U}_n$; so for each $U^{/}_n \in \mathcal{U}^{/}_n$, there exists $U_n \in \mathcal{U}_n$ such that $\psi_{2}-cl(U^{/}_n) \subseteq U_n$, for each $n\in \mathbb{N}$. Now if we fix $n$, then for each $V \in \mathcal{V}^{/}_n$, there exists $U_{V} \in \mathcal{U}_n$ such that $\psi_{2}-cl(V) \subseteq U_{V}$. This implies that $\bigcup_{V \in \mathcal{V}^{/}_n} \psi_{2}-cl(V) \subseteq \bigcup_{V \in \mathcal{V}^{/}_n} U_{V}$.  Then, $\bigcup_{n \in \mathbb{N}}(\bigcup_{V \in \mathcal{V}^{/}_n} \psi_{2}-cl(V)) \subseteq \bigcup_{n \in \mathbb{N}}(\bigcup_{V \in \mathcal{V}^{/}_n} U_{V})$. Thus, $Y \subseteq \bigcup_{n \in \mathbb{N}}(\bigcup_{V \in \mathcal{V}^{/}_n} U_{V})$. Also, $\bigcup_{n \in \mathbb{N}}(\bigcup_{V \in \mathcal{V}^{/}_n} U_{V}) \subseteq Y$. This gives $Y = \bigcup_{n \in \mathbb{N}}(\bigcup_{V \in \mathcal{V}^{/}_n} U_{V})$. For each $n \in \mathbb{N}$, we take $\mathcal{V}_n =\{U_{V}:V \in \mathcal{V}^{/}_n\}$. Then, $ \bigcup_{n \in \mathbb{N}}\mathcal{V}_n$ covers $Y$. Now since each $\mathcal{V}^{/}_n$ is a finite set, so for each $n\in \mathbb{N}$, $\mathcal{V}_n$ is a finite subset of $\mathcal{U}_n$. 

Also for each $n\in \mathbb{N}$, $H(x,1)\in \bigcup \mathcal{V}^{/}_n$. This implies that $H(x,1)\in V $, for some $V \in \mathcal{V}^{/}_n$. This gives 
$H(x,1)\in \psi_{2}-cl(V) \subseteq U_{V} $, $U_{V} \in \mathcal{U}_n$. Thus, $H(x,1)\in \bigcup \mathcal{V}_n$, for each $n \in \mathbb{N}$.

Now, since $Y$ is pairwise $T_3$ and pairwise locally compact, so it is pairwise regular, pairwise Hausdorff and pairwise locally compact. Then by proposition 2.2, $\psi_{1} \subseteq \psi_{2}$. This implies that $\bigcup\atop_{n \in \mathbb{N}}$ $\mathcal{V}_n\in \mathcal{O}_{2}$. Thus, $(Y,\psi_{1} ,\psi_{2})$ is $(1,2)$-H-Menger.
\end{proof}
\begin{theorem}
Let $(X,f)$ and $(Y,g)$ be two $BTDSs$, where $(X,\tau_{1} ,\tau_{2})$ and $(Y,\psi_{1} ,\psi_{2})$ are two bitopological spaces. If  $f\approx^{H}_F g$ and if $(Y,\psi_{1} ,\psi_{2})$ is a pairwise $P$-space, then
    $Y$ has H-almost Menger property iff $Y$ has H-weak Menger property.
\end{theorem}
\begin{proof}
The necessary part follows from theorem 4.1. For the sufficient part, suppose that $Y$ has H-weak Menger property. Then, $Y$ has both $(1,2)$-H-weak Menger property and $(2,1)$-H-weak Menger property. We shall prove that if $Y$ has $(1,2)$-H-weak Menger property, then $Y$ has $(1,2)$-H-almost Menger  property. The proof of $(2,1)$-H-weak Menger property implies $(2,1)$-H-almost Menger  property will be similar.

Suppose that $\mathcal{X}=(\mathcal{U}_n: n\in \mathbb{N})$ is an arbitrary sequence of $\psi_{1}$-open covers of $Y$ and that $x \in X$ is an arbitrary point such that $H(x,0)\in \bigcup \mathcal{U}_n$, for each $n \in \mathbb{N}$. Now, $(1,2)$-H-weak Menger property of $Y$ implies that there is a sequence $\mathcal{Y}=(\mathcal{V}_n: n\in \mathbb{N})$ such that for each $n\in \mathbb{N}$, $\mathcal{V}_n$ is a finite subset of $\mathcal{U}_n$ and $H(x,1)\in \bigcup \mathcal{V}_n$, $\mathcal{V}_n\in \mathcal{Y}$, for each $n \in \mathbb{N}$ and $\bigcup_{n \in \mathbb{N}} \mathcal{V}_n \in \mathcal{O}^D_{12}$. This gives $Y = \psi_{2}-cl(\bigcup_{n\in \mathbb{N}} \mathcal{V}_n)$. Now since each $\mathcal{V}_n$ is a finite set, so $\bigcup_{V \in \mathcal{V}_n} \psi_{2}-cl(V)$ is finite union of $\psi_2$-closed set and hence $\psi_2$-closed. Since $Y$ is a pairwise $P$-space, so $\bigcup_{n \in \mathbb{N}} (\bigcup_{V \in \mathcal{V}_n} \psi_{2}-cl(V))$ is a $\psi_2$-closed set. Now if we fix n, then for each $V \in \mathcal{V}_n$ we have $V \subseteq \psi_{2}-cl(V)$. Thus,  $\bigcup_{n \in \mathbb{N}} (\bigcup_{V \in \mathcal{V}_n} \psi_{2}-cl(V))$ is a $\psi_2$-closed set containing $\bigcup_{n\in \mathbb{N}} \mathcal{V}_n$. But $\psi_{2}-cl(\bigcup_{n\in \mathbb{N}} \mathcal{V}_n)$ is the smallest $\psi_2$-closed set containing $\bigcup_{n\in \mathbb{N}} \mathcal{V}_n$. Hence, $\psi_{2}-cl(\bigcup_{n\in \mathbb{N}} \mathcal{V}_n) \subseteq \bigcup_{n \in \mathbb{N}} (\bigcup_{V \in \mathcal{V}_n} \psi_{2}-cl(V))$. This gives $Y \subseteq \bigcup_{n \in \mathbb{N}} (\bigcup_{V \in \mathcal{V}_n} \psi_{2}-cl(V))$. Also, $ \bigcup_{n \in \mathbb{N}} (\bigcup_{V \in \mathcal{V}_n} \psi_{2}-cl(V)) \subseteq Y$. Hence, $Y=\bigcup_{n \in \mathbb{N}} (\bigcup_{V \in \mathcal{V}_n} \psi_{2}-cl(V))$. Thus, $\bigcup\atop_{n \in \mathbb{N}}$ $\mathcal{V}_n\in \overline{\mathcal{O}_{12}}$. Hence, $Y$ is $(1,2)$-H-almost Menger.  
\end{proof}
Theorem 4.10. and theorem 4.11. together gives the following corollary.
\begin{corollary}
Let $(X,f)$ and $(Y,g)$ be two $BTDSs$, where $(X,\tau_{1} ,\tau_{2})$ and $(Y,\psi_{1} ,\psi_{2})$ are two bitopological spaces. If  $f\approx^{H}_F g$ and if $(Y,\psi_{1} ,\psi_{2})$ is a pairwise $T_3$ and pairwise locally compact pairwise $P$-space, then
    $Y$ has H-Menger property iff $Y$ has H-weak Menger property.
\end{corollary}
\begin{theorem} Let $(X,f)$ and $(Y,g)$ be two $BTDSs$, where $(X,\tau_{1} ,\tau_{2})$ and $(Y,\psi_{1} ,\psi_{2})$ are two bitopological spaces. If  $f\approx^{H}_F g$ and if $(Y,\psi_{1},\psi_{2})$ is a pairwise $T_3$ and pairwise locally compact pairwise $P$-space, then following conditions are equivalent:
\begin{enumerate}
    \item[(i)] $Y$ has H-Menger property,
    \item[(ii)] $Y$ has H-almost Menger property,
    \item[(iii)] $Y$ has H-weak Menger property.
\end{enumerate}
\end{theorem}
\begin{theorem} Let $([0,1],f)$ and $(X,g)$ be two $BTDSs$, where $([0,1],\tau_{1} ,\tau_{2})$ and $(X,\psi_{1} ,\psi_{2})$ are two bitopological spaces. If  $f\approx^{H}_F g$ and if $(Y,\psi_{1},\psi_{2})$ is a pairwise $T_3$ and pairwise locally compact pairwise $P$-space, then following conditions are equivalent:
\begin{enumerate}
    \item[(i)] $X$ has PH-Rothberger property,
    \item[(ii)] $X$ has PH-almost Rothberger property,
    \item[(iii)] $X$ has PH-weak Rothberger property.
\end{enumerate}
\end{theorem}
\begin{theorem} Let $([0,1],f)$ and $(X,g)$ be two $BTDSs$, where $([0,1],\tau_{1} ,\tau_{2})$ and $(X,\psi_{1} ,\psi_{2})$ are two bitopological spaces. If  $f\approx^{H}_F g$ and if $(Y,\psi_{1},\psi_{2})$ is a pairwise $T_3$ and pairwise locally compact pairwise $P$-space, then following conditions are equivalent:
\begin{enumerate}
    \item[(i)] $X$ has PH-Menger property,
    \item[(ii)] $X$ has PH-almost Menger property,
    \item[(iii)] $X$ has PH-weak Menger property.
\end{enumerate}
\end{theorem}

\section{Conclusion}

 In this paper, we introduced homotopy in bitopological dynamical systems via selection principles. First, we define $BTDS-$homotopy, $BTDS-$iteration homotopy and $BTDS-$path homotopy and then using this notions we define various selection properties viz.  H-Rothberger property, H-Menger property, PH-Rothberger property, PH-Menger property and their weaker versions. We also discuss several results connecting this concepts. We also introduced pairwise $P$-space in bitopology. Currently, homotopy theory in topological space has been applied to many different fields. We hope that our theory of homotopy in bitopological dynamical systems via selection principles will become a base for future development and application from the perspective of bitopology as well as from bitopological dynamical systems. \\ \vskip 5mm
{\bf Acknowledgement:} The authors like to thank for Prof. Ljubisa Kocinac,  University of Nis, Serbia  for his constructive suggestions and discussions during the preparation of this paper.\\

{\bf Conflict of interest:} The authors declare that there is no conflict of interest. 

\end{document}